\documentclass{article}

\usepackage{hyperref}
\hypersetup{colorlinks=true,
	linkcolor=black, 
	urlcolor=blue} 
\usepackage{amsmath,amssymb,amsfonts,amsthm}
\usepackage{amsthm}
\usepackage{enumerate}
\numberwithin{equation}{section}
\usepackage{marginnote}
\usepackage[a4paper]{geometry}

\reversemarginpar

\usepackage[english]{babel}
\usepackage[babel, english=american]{csquotes}

\newtheoremstyle{Teorema}
{3pt}
{3pt}
{\slshape}
{}
{\bfseries}
{:}
{\newline}
{}

\newtheorem{theorem}{Theorem}[section]
\newtheorem{definition}{Definition}

\newtheorem{corollary}[theorem]{Corollary}

\newtheorem{lemma}[theorem]{Lemma}

\theoremstyle{definition}

\DeclareMathOperator{\lcm}{lcm}
\DeclareMathOperator{\D}{d}
\DeclareMathOperator{\supp}{supp}
\DeclareMathOperator{\Co}{Co}

\newcommand{\N}{\mathbb{N}}
\newcommand{\Z}{\mathbb{Z}}
\newcommand{\R}{\mathbb{R}}
\newcommand{\C}{\mathbb{C}}

\newcommand{\diff}{\backslash}

\title{On the convergence of multi-level Hermite-Pad\'e approximants for a class of meromorphic functions}
\author{L. G. Gonz\'alez Ricardo, G. L\'opez Lagomasino,\\ and S. Medina Peralta}
\date{\today}

\begin{document}

\maketitle
\begin{abstract}

	The present paper deals with the convergence properties of multi-level Hermite-Pad\'e approximants for a class of meromorphic functions given by rational perturbations with real coefficients of a Nikishin system of functions, and study the zero location of the corresponding multiple orthogonal polynomials.
\end{abstract}

\noindent \textbf{Keywords:} Nikishin system, multiple orthogonal polynomials, Hermite-Pad\'e approximation

\noindent \textbf{AMS subject classification:} Primary: 30E10, 41A21; Secondary: 42C05

\section{Introduction}

This paper deals with the proof of Markov \cite{markov} and Stieltjes \cite{stieltjes} type theorems for the convergence of Hermite-Pad\'e approximants of certain type meromorphic functions. In the context of Pad\'e approximants this study was initiated in \cite{gonchar1} (see also \cite{lago81}, \cite{Lago89} and references therein).

Hermite-Pad\'e aproximation and their associated multiple orthogonal polynomials have received much attention in recent years. This is partly due to the many areas in which they have been found to be useful. Such areas cover, but are not limited to, number theory,  non-intersecting brownian motions, multiple orthogonal polynomial ensembles, random matrix theory, and differential equations. For specific references see \cite{Lago_Sergio_Jacek}. In particular, in \cite{LS} they were used to find discrete (peakon) solutions of the Degasperis-Procesi partial differential equation (see also references therein). The problem was reduced to finding a pair of discrete measures whose mass points and corresponding masses fully characterize the peakon solutions. An appropriate interpolation problem was defined to find the Cauchy transforms of such measures ant thus the measures themselves. We wish to point out that the pair of measures appearing in the problem form a generator of a Nikishin system (see definition in next subsection). For more details see \cite{LS} and \cite[Appendix]{Lago_Sergio_Jacek}.

Motivated by \cite{LS}, in \cite{Lago_Sergio_Jacek} we studied (and proved) the convergence of such interpolation processes for the case of general Nikishin systems whose generating measures were continuous. Here, we wish to see what happens when the Nikishin system is perturbed with rational functions with real coefficients. A similar question was raised and solved in \cite{Lago_Sergio1} for type I Hermite-Pad\'e approximation (see also \cite{Ulises_Lago_Sergio1}).

\subsection{Nikishin systems}

Nikishin systems were first introduced in \cite{nikishin} and were initially named MT-systems. Here, we will slightly extend the definition to include measures with unbounded support.

In the sequel, we will only consider Borel measures $s$ with constant sign, finite moments $c_n = \int x^n \D s (x), |c_n| < \infty$, $n\in\Z_+$, whose support $\supp s$ consists of infinitely many points, and is contained in $\R$. We will denote by $\Delta$ the smallest interval  containing $\supp s$, i.e. its convex hull. This class of these measures will be denoted by $\mathcal{M}(\Delta)$. Let
$$\widehat{s}(z) = \int \frac{\D s(x)}{z-x}$$
denote the Cauchy transform of the measure $s$. Obviously, $\widehat{s}$ is holomorphic in $\C\diff\Delta$ and we can associate to $\widehat{s}$ its formal Taylor expansion at infinity
$$\widehat{s}(z)\sim \sum_{j=0}^\infty\frac{c_j}{z^{j+1}},\qquad c_j = \int x^j\D s(x).$$
We say that the measure $s$ satisfies Carleman's condition when
\[ \sum_{n=0}^\infty \frac{1}{|c_n|^{1/2n}} = \infty.
\]

Let $\Delta_\alpha$, $\Delta_\beta$ be two intervals contained in the real line such that $\Delta_\alpha\cap\Delta_\beta=\emptyset$. Consider the measures $\sigma_\alpha\in \mathcal{M}(\Delta_\alpha), \sigma_\beta\in\mathcal{M}(\Delta_\beta), \widehat{\sigma}_\beta \in L_1(\sigma_\alpha)$.   With these two measures  we construct a third one as follows (using differential notation)
$$\D\langle \sigma_\alpha,\sigma_\beta\rangle(x): = \widehat{\sigma}_\beta(x)\D\sigma_\alpha(x).$$

When we consider consecutive products of measures, a.e. $\langle \sigma_\alpha,\sigma_\beta, \sigma_\gamma \rangle := \langle\sigma_\alpha, \langle \sigma_\beta, \sigma_\gamma \rangle \rangle$ we implicitly assume not only that $\widehat{\sigma}_\gamma \in L_1(\sigma_\beta)$, but also $\langle \sigma_\beta, \sigma_\gamma\widehat{\rangle}\in L_1(\sigma_\alpha)$, where $\langle \sigma_\beta, \sigma_\gamma\widehat{\rangle}$ denotes the Cauchy transform of $\langle \sigma_\beta, \sigma_\gamma\rangle$. It is important to remark that this product is neither commutative nor associative.

\begin{definition}
\label{Nikishin_sys}
Take a collection $\Delta_j$, $j=1,\ldots,m$ of intervals such that
$$\Delta_j\cap\Delta_{j+1}=\emptyset,\qquad j=1,\ldots,m-1.$$
Let $(\sigma_1,\ldots,\sigma_m)$ be a system of measures such that $\Co(\supp \sigma_j)=\Delta_j$, $\sigma_j\in\mathcal{M}(\Delta_j)$, $j=1,\ldots,m$. We say that $(s_{1,1},\ldots,s_{1,m})=\mathcal{N}(\sigma_1,\ldots,\sigma_m)$, where
$$s_{1,1}=\sigma_1,\quad s_{1,2}=\langle\sigma_1,\sigma_2 \rangle,\; \ldots, \quad s_{1,m}=\langle \sigma_1, \langle \sigma_2,\ldots,\sigma_m \rangle\rangle,$$
is the Nikishin system of measures generated by $(\sigma_1,\ldots,\sigma_m)$. The vector $(\widehat{s}_{1,1},\ldots,\widehat{s}_{1,m})$ is called a Nikishin system of functions.
\end{definition}

Notice that any sub-system of $(\sigma_1,\ldots,\sigma_m)$ of consecutive measures is also a generator of some Nikishin system. In the sequel for $1\leq j\leq k\leq m$ we will write
$$s_{j,k} := \langle \sigma_j, \sigma_{j+1},\ldots, \sigma_k \rangle,\qquad s_{k,j} := \langle \sigma_k, \sigma_{k-1},\ldots, \sigma_j \rangle.$$
In particular, with the system of measures $(\sigma_1,\ldots,\sigma_m)$ we can also define the reversed Nikishin system $(s_{m,m}, \ldots, s_{m,1}) = \mathcal{N}(\sigma_m,\ldots,\sigma_1)$ which plays a significant role in the sequel.

\subsection{Statement of the main result}

Let us start defining the approximation objects.
\begin{definition}
Consider the Nikishin system $\mathcal{N}(\sigma_1,\ldots,\sigma_m)$. Let $\displaystyle r_j = \frac{v_j}{t_j}$, $k=1,\ldots,m$, be rational fractions with real coefficients, $\deg v_k<\deg t_k=d_k$, $(v_k,t_k)=1$ (coprime) for all $k=1,\ldots,m$. For each $n\in\N$, there exist polynomials $a_{n,0}, a_{n,1},\ldots,a_{n,m}$, with $\deg a_{n,j}\leq n-1$, $j=0,1,\ldots,m-1, \deg a_{n,m}\leq n$, not all identically equal to zero, called multi-level (ML) Hermite-Pad\'e polynomials that verify
\begin{align}
    \mathcal{A}_{n,0} :=& \left[ a_{n,0} + \sum_{k=1}^m (-1)^k a_{n,k}(\widehat{s}_{1,k}+r_k) \right] \in \mathcal{O}\left(\frac{1}{z^{n+1}}\right), \label{Problem1}\\
    \mathcal{A}_{n,j} :=& \left[(-1)^j a_{n,j} + \sum_{k=j+1}^m (-1)^k a_{n,k}\widehat{s}_{j+1,k}\right] \in \mathcal{O}\left(\frac{1}{z}\right), \quad j=1,\ldots,m-1.\label{Problem2}
\end{align}
Here and in the sequel $\mathcal{O}(\cdot)$ is as $z \to \infty$ along paths non tangential to the support of the measures involved. For completeness we denote $\mathcal{A}_{n,m} := (-1)^m a_{n,m}$.
\end{definition}

When $r_k \equiv 0, k=1,\ldots,m,$ this construction was introduced in \cite{Lago_Sergio_Jacek}.
Notice that in this scheme of approximation the interpolation conditions involve all Nikishin systems of the ``inner levels'', i.e. $\mathcal{N}(\sigma_1,\ldots,\sigma_m)$, $\mathcal{N}(\sigma_2,\ldots,\sigma_m)$, \ldots, $\mathcal{N}(\sigma_m) = (s_{m,m})$. Finding the polynomials $a_{n,0}, a_{n,1},\ldots,a_{n,m}$ is equivalent to solving a homogeneous linear system of $n(m+1)$ equations, given by the interpolation conditions, on $n(m+1)+1$ unknowns, corresponding to the coefficients of the polynomials. Consequently, the system of equations has a non trivial solution. However, the solution need not be unique.

Let $T = \lcm (t_1,\ldots,t_m), \deg T = D$,   where $\lcm$ stands for least common multiple.

\begin{theorem}
\label{Th_conv_unif}
For each $n\in\N$ let $a_{n,0}, a_{n,1},\ldots, a_{n,m}$ be Hermite-Pad\'e polynomials associated with the Nikishin system $\mathcal{N}(\sigma_1,\ldots,\sigma_m)$ and $(r_1,\ldots,r_m)$ such that (\ref{Problem1}) and (\ref{Problem2}) holds. Assume that the zeros of the polynomial $T$ lie in the complement of $\Delta_1 \cup \Delta_m$ and $f$ has exactly $D$ poles in $\C \setminus \Delta_m$, where
\[ f := \widehat{s}_{m,1} - \sum_{k=1}^{m-1} (-1)^{k}\widehat{s}_{m,k+1}r_k - (-1)^m r_m.
\]
Suppose that either the sequence of moments of $\sigma_m$ satisfies Carleman's conditions or $\Delta_{m-1}$ is a bounded interval. Then,
\begin{equation*}
\lim_n\frac{a_{n,j}}{a_{n,m}} = \widehat{s}_{m,j+1}, \qquad j=1,\ldots,m-1,
\end{equation*}
and
\begin{equation} \label{polos}
  \lim_n \frac{a_{n,0}}{a_{n,m}} = f(z),
\end{equation}
uniformly on each compact subset of $\C\diff(\Delta_m \cup \{z:T(z) = 0\})$. For all sufficiently large $n$, $\deg a_{n,m} = n,$  $a_{n,m}$ has exactly $n-D$ simple zeros in the interior of $\Delta_m$ and $D$ zeros in $\C \setminus \Delta_m$ which converge to the poles of $f$ in this region according to their order. For $j=1,\ldots,m-1$ and all sufficiently large $n$ the polynomial $a_{n,j}$ has at least $n-D -m +j$ sign changes in $\Delta_m$ and at least $D$ zeros in $\C \setminus \Delta_m$ of which $D$ converge to the zeros of $T$ according to their multiplicity and the remaining ones accumulate on $\Delta_m \cup \{\infty\}$.
\end{theorem}

The fact that $\deg a_{n,m} = n$ for all sufficiently large $n$ implies that for such indices the vector $(a_{n,0},\ldots,a_{n,m})$ is uniquely determined except for a constant multiple of it. Indeed, from two non-collinear  solutions of \eqref{Problem1}-\eqref{Problem2} one can construct a non-trivial solution whose last polynomial has degree smaller that $n$.

Notice that $f(z) \equiv \widehat{s}_{m,1}$ when $r_k \equiv 0, k=1,\ldots,m,$  and Theorem \ref{Th_conv_unif} gives the main statement in  \cite{Lago_Sergio_Jacek}; namely, relation (1.23) of Theorem 1.6. The expression of the limit relations in Theorem \ref{Th_conv_unif}  are similar to those in \cite[Theorem 1.2]{Lago_Sergio1} where type I Hermite-Pad\'e approximants of meromorphic functions were studied.

Obviously, the poles of $f$ in $\C \setminus \Delta_m$ are the zeros of $T$. Therefore, the total number of poles of $f$ in that region equals $D$ if and only if for each zero $\zeta$ of $T$, say of multiplicity $\tau$, we have
\[ \lim_{z \to \zeta} (z-\zeta)^{\tau}f(z) =   - \sum_{k=1}^{m-1} (-1)^{k}\widehat{s}_{m,k+1}(\zeta) \lim_{z\to \zeta} (z-\zeta)^\tau r_k(z) - (-1)^m \lim_{z\to \zeta} (z-\zeta)^\tau r_m(z) \neq 0.
\]
Therefore, sufficient conditions for $f$ to have $D$ poles in $\C \setminus \Delta_m$ is that $(t_j,t_k) = 1, 1\leq j,k \leq m$ or, more generally, that for each $\zeta, T(\zeta) = 0,$ there is only one polynomial $t_k$ which has $\zeta$ as zero of degree $\tau$. Indeed,   in this case all the terms in the previous sum cancel except one which is trivially different from zero. (The functions $\widehat{s}_{m,j}, j=1,\ldots,m,$ are never zero in $\C \setminus \Delta_m$.)

\section{Auxiliary results and concepts}


In this section we introduce some necessary definitions and results needed in our developments. We start with a useful Lemma whose proof in the case of measures with bounded support is an easy consequence of Cauchy's integral formula and Fubini's theorem but in the unbounded case is more elaborate and may be found in \cite[Lemma 2.1]{Lago_Sergio}.

\begin{lemma}
\label{Remainder_type_i}
Let $(s_{1,1},\ldots,s_{1,m}) =\mathcal{N}(\sigma_1,\ldots,\sigma_m)$ be given. Assume that there exist polynomials with real coefficients $a_0,\ldots, a_m$ and a polynomial $w$ with real coefficients whose zeros lie in $\C\backslash\Delta_1$ such that
\begin{equation*}
    \frac{\mathcal{A}(z)}{w(z)}\in\mathcal{H}(\C\backslash\Delta_1) \;\textrm{ and }\; \frac{\mathcal{A}(z)}{w(z)}=\mathcal{O}\left(\frac{1}{z^N}\right),\quad z\rightarrow \infty,
\end{equation*}
where $\displaystyle \mathcal{A}:= a_0 +\sum_{k=1}^m a_k\widehat{s}_{1,k}$ and $N\geq 1$. Let $\displaystyle \mathcal{A}_1:= a_1 + \sum_{k=2}^m a_k\widehat{s}_{2,k}$. Then,
\begin{equation} \label{relint}
    \frac{\mathcal{A}(z)}{w(z)} = \int \frac{\mathcal{A}_1(x)}{z-x}\frac{\D \sigma_1(x)}{w(x)}.
\end{equation}
If $N\geq 2$, we also have
\begin{equation}
    \int x^\nu \mathcal{A}_1(x)\frac{\D \sigma_1(x)}{w(x)}=0,\qquad \nu=0,1,\ldots, N-2. \label{Orth_remainder}
\end{equation}
In particular, $\mathcal{A}_1$ has at least $N-1$ sign changes in $\mathring{\Delta}_1$ (the interior of $\Delta_1$ in $\R$ with the usual topology).
\end{lemma}

In the following, we need some relations involving reciprocals and ratios of Cauchy transforms of measures. It is well known that for each measure $\sigma\in\mathcal{M}(\Delta)$, where $\Delta$ is contained in a half line (by a half line we mean an interval of the form $[c,+\infty)$ or $(-\infty,c], c \in \R$), there exist a measure $\tau\in\mathcal{M}(\Delta)$ and a polynomial $\ell(z) = az+b$, $a=1/|\sigma|,b\in\R$, such that
\begin{equation*}
    \frac{1}{\widehat{\sigma}(z)} = \ell(z) + \widehat{\tau}(z),
\end{equation*}
where $|\sigma|$ is the total variation of the measure $\sigma$. For more information in the case of measures with compact support see \cite[Appendix]{Krein_Nudelman} and \cite[Theorem 6.3.5]{Stahl_Totik}, when the measure is supported in a half line see \cite[Lemma 2.3]{ulises_lago_2}. If $\sigma$ satisfies Carleman's condition $\sum_{n=0}^{\infty}|c_n|^{-1/2n}=\infty$ then $\tau$ satisfies the same condition, \cite[Theorem 1.5]{Lago_Sergio}. We call $\tau$ the inverse measure of $\sigma$.

Such measures appear frequently in our arguments, so we will fix a notation to differentiate them. In relation with the measures denoted with $s$ they will carry over to them the corresponding sub-indices. The same goes for the polynomials $\ell$. For example,
\begin{equation}\label{CauchyTranInver}
    \frac{1}{\widehat{s}_{j,k}(z)} = \ell_{j,k}(z) + \widehat{\tau}_{j,k}(z).
\end{equation}
We also use
\begin{equation*}
    \frac{1}{\widehat{\sigma}_\alpha(z)} = \ell_\alpha(z) + \widehat{\tau}_\alpha(z).
\end{equation*}
On some occasions we write $\langle \sigma_\alpha,\sigma_\beta \widehat{\rangle}$ in place of $\widehat{s}_{\alpha,\beta}$. In the paper \cite[Lemma 2.10]{Ulises_Lago} (see also \cite{ulises_lago_2}) several formulas involving Cauchy transforms of measures were proved. For our reasonings, the most important ones establish that
\begin{equation}
    \frac{\widehat{s}_{1,k}}{\widehat{s}_{1,1}} = \frac{|s_{1,k}|}{|s_{1,1}|} - \langle \tau_{1,1}, \langle s_{2,k}, \sigma_1 \rangle \widehat{\rangle}, \label{QuotCauchyTran}
\end{equation}
where $|s|$ denotes the total variation of the measure $s$.

Another important notion for the proofs to come is the convergence in Hausdorff content. Let $A$ be a subset of $\C$. By $\mathcal{U}(A)$ we denote the class of all coverings of $A$ by at most a numerable set of disks. Set
\begin{equation*}
    h(A) = \inf\left\{ \sum_{i=1}^\infty |U_i| \mid \{U_i\}\in\mathcal{U}(A)\right\},
\end{equation*}
where $|U_i|$ stands for the radius of the disk $U_i$. The quantity $h(A)$ is called the 1-dimensional Hausdorff content of the set $A$.

Let $\{\phi_n\}_{n\in\N}$ be a sequence of complex functions defined on a region $D\subset\C$ and $\phi$ another function defined on $D$ (the value $\infty$ is permitted). We say that $\{\phi_n\}_{n\in\N}$ converges in Hausdorff content to the function $\phi$ inside $D$ if for each compact subset $K$ of $D$ and for each $\varepsilon>0$, we have
    \begin{equation*}
        \lim_{n\to \infty} h\{ z\in K : |\phi_n(z)-\phi(z)|>\varepsilon \} =0
    \end{equation*}
    (by convention $\infty\pm\infty=\infty$). We denote this writing $h-\lim_{n\to \infty} \phi_n=\phi $ inside $D$.

As usual, we denote the space of analytic functions over a region $\Omega$ in the complex plain by $\mathcal{H}(\Omega)$. In order to obtain the convergence of the approximants in Hausdorff content, we need the notion of incomplete multi-point Pad\'e approximant.
\begin{definition}
\label{Incomplete_Pade}
Let $s\in\mathcal{M}(\Delta)$ where $\Delta$ is contained in a half line of the real axis. Fix an arbitrary $\kappa\geq-1$. Consider a sequence of polynomials $\{\omega_n\}_{n\in\Lambda}$, $\Lambda\subset\Z_+$, such that $\deg \omega_n=\kappa_n\leq 2n+\kappa+1$, whose zeros lie in $\R\backslash\Delta$. Let $R_n=p_n/q_n$ be a sequence of rational functions with real coefficients such that for each $n\in\Lambda$:
\begin{enumerate}[a)]
    \item $\deg p_n\leq n+\kappa$, $\deg q_n\leq n$, $q_n\not\equiv 0$,
    \item $\displaystyle \frac{q_n\widehat{s}-p_n}{\omega_n}(z)=\mathcal{O}\left(\frac{1}{z^{n+1-l}}\right)\in\mathcal{H}(\C\backslash\Delta), z\rightarrow \infty$, for some fixed $l\in\Z_+$.
\end{enumerate}
We say that $\{R_n\}_{n\in\Lambda}$ is a sequence of incomplete diagonal multi-point Pad\'e approximants of $\widehat{s}$.
\end{definition}

For sequences of incomplete diagonal multi-point Pad\'e approximants, the following Stieltjes type theorem was proved in \cite[Lemma 2]{Bust_Lago} in terms of convergence in logarithmic capacity.  Using Hausdorff content the proof is basically the same and in fact simpler since the Hausdorff content of a set is easier to estimate than its logarithmic capacity.

\begin{lemma}
\label{Conv_MP_Incom_Pade}
Let $s\in\mathcal{M}(\Delta)$ be given, where $\Delta$ is contained in a half line. Assume that $\{R_n\}_{n\in\N}$ satisfies \emph{a)-b)} and either the number of zeros of $\omega_n$ lying on a bounded segment of $\R\backslash\Delta$ tends to infinity as $n\rightarrow\infty$, $n\in\Lambda$, or $s$ satisfies Carleman's condition. Then
\begin{equation*}
    h-\lim_{n\in\Lambda} R_n = \widehat{s},\quad \textrm{ inside } \quad \C\backslash\Delta.
\end{equation*}
\end{lemma}

The proof of Theorem \ref{Th_conv_unif} goes as follows. Instead of proving convergence with  the uniform norm we will establish similar results but in Hausdorff content. On the other hand, we study the location of the zeros of the polynomials $a_{n,m}$. With this additional information and a very useful lemma of A. A. Gonchar \cite[Lemma 1]{gonchar} we can derive convergence in the uniform norm from convergence in Hausdorff content.

\section{Proof of main result and consequences}

\subsection{General properties of zeros}

Our first result is related with the location of the zeros of the polynomials $a_{n,j}$ and the forms $\mathcal{A}_{n,j}$. As above $T = \lcm(t_1,\ldots,t_m) $ and $D=\deg T$.

\begin{lemma} \label{Zeros_a_nm}
For each $n \geq 2D$, the form $\mathcal{A}_{n,j}, j=1,\ldots,m,$ has at least $n-2D$ sign changes in $\mathring{\Delta}_j$ and at most $n$ zeros in $\C \setminus \Delta_{j+1}$ ($\Delta_{m+1} = \emptyset$). If the zeros of $T$ lie outside of $\Delta_1$ then $\mathcal{A}_{n,j}, j=1,\ldots,m,$ has at least $n-D$ sign changes in $\mathring{\Delta}_j$. The form $\mathcal{A}_{n,0}$ has at most $2D$ zeros in $\C \setminus \Delta_1$ and this quantity reduces to $D$ should the zeros of $T$ lie in the complement of $\Delta_1$. If the zeros of $T$ lie outside $\Delta_1$ and for some $n$ we know that $a_{n,m}$ has exactly $n-D$ sign changes on $\Delta_m$ then, $\mathcal{A}_{n,0}$ cannot have zeros in $\C \setminus \Delta_1$ and $\mathcal{A}_{n,j}, j=1,\ldots,m-1$   has exactly $n-D$ zeros in $\C \setminus \Delta_{j+1}$ they are all simple and lie on $\Delta_j$.
\end{lemma}
\begin{proof}  Fix $n \geq 2D$. Consider the linear form
\begin{align*}
    \mathcal{L}_{n,0}(z) :=& T(z)\mathcal{A}_{n,0}(z) = \left[a_{n,0}T + \sum_{k=1}^m (-1)^k a_{n,k}Tr_k + \sum_{k=1}^m (-1)^k a_{n,k}T\widehat{s}_{1,k}\right](z)\nonumber\\
                         =& \left[p_{n,0} + \sum_{k=1}^m (-1)^k p_{n,k}\widehat{s}_{1,k}\right](z) = \mathcal{O}\left(\frac{1}{z^{n-D+1}}\right),
\end{align*}
where
\begin{equation}\label{ps}
    p_{n,0} = a_{n,0}T + \sum_{k=1}^m (-1)^k a_{n,k}Tr_k,\qquad p_{n,k} =   a_{n,k}T,\quad k=1,\ldots,m.
\end{equation}

Using Lemma \ref{Remainder_type_i}, in particular (\ref{Orth_remainder}), we obtain the following orthogonality relations
\begin{equation*}
    \int x^\nu \mathcal{L}_{n,1}(x)\D \sigma_1(x) =0,\qquad \nu = 0,1,\ldots,n-D-1,
\end{equation*}
and $\mathcal{L}_{n,1}:= - p_{n,1} + \sum_{k=2}^m (-1)^k p_{n,k}\widehat{s}_{2,k}$ has at least $n-D$ sign changes on $\mathring{\Delta}_1$.

Notice that
\begin{equation*}
    \mathcal{L}_{n,1} = - p_{n,1} + \sum_{k=2}^m (-1)^k p_{n,k}\widehat{s}_{2,k}
                         = -Ta_{n,1} + \sum_{k=2}^m (-1)^k Ta_{n,k}\widehat{s}_{2,k}
                         = \mathcal{A}_{n,1}T.
\end{equation*}
Therefore, $\mathcal{A}_{n,1}$ has at least $n-2D$ sign changes in the interior of $\Delta_1$ ($D$ sign changes may be on account of $T$). However, if the zeros of $T$ are in the complement of $\Delta_1$ then we can affirm that $\mathcal{A}_{n,1}$ has at least $n-D$ sign changes in the interior of $\Delta_1$. These two situations are accountable for the different statements on the number of sign changes of $\mathcal{A}_{n,j}$ on $\Delta_j$.

Let $w_{n,1}$ be a polynomial with simple zeros at the points of sign change of $\mathcal{A}_{n,1}$ on $\mathring{\Delta}_1$. In general $\deg w_{n,1} \geq n-2D$, but $\deg w_{n,1} \geq n-D$ if the zeros of $T$ lie outside $\Delta_1$. Therefore,
\begin{equation*}
     \mathcal{H}(\C\diff\Delta_2)\ni\frac{\mathcal{A}_{n,1}}{w_{n,1}} = \mathcal{O}\left(\frac{1}{z^{ \deg( w_{n,1})+1}}\right).
\end{equation*}

Notice that $\mathcal{A}_{n,1}$ and $w_{n,1}$ satisfy the hypothesis of Lemma \ref{Remainder_type_i}, so
\begin{equation}
    \frac{\mathcal{A}_{n,1}(z)}{w_{n,1}(z)} = \int \frac{\mathcal{A}_{n,2}(x)}{z-x}\frac{\D\sigma_2(x)}{w_{n,1}(x)},\nonumber
\end{equation}
and
\begin{equation}
    \int x^\nu \mathcal{A}_{n,2}(x)\frac{\D\sigma_2(x)}{w_{n,1}(x)}=0,\qquad \nu = 0,1,\ldots, \deg(w_{n,1})-1.\nonumber
\end{equation}
This yields that $\mathcal{A}_{n,2}$ has at least $\deg(w_{n,1})$ sign changes in the interior of $\Delta_2$.

Again, let $w_{n,2}$ be a polynomial with simple zeros at the points of sign change of $\mathcal{A}_{n,2}$ in $\Delta_2$. Hence, $\deg (w_{n,2}) \geq \deg(w_{n,1})$ and
\begin{equation*}
    \mathcal{H}(\C\diff\Delta_3)\ni\frac{\mathcal{A}_{n,2}}{w_{n,2}} = \mathcal{O}\left(\frac{1}{z^{\deg(w_{n,1})+1}}\right).
\end{equation*}
Then, we have deduced the same conclusions for $\mathcal{A}_{n,2}$ as  we had for $\mathcal{A}_{n,1}$, and we can repeat the same reasonings inductively obtaining that for each $j=1,\ldots,m-1$ there exists a polynomial $w_{n,j}, \deg (w_{n,j}) \geq \deg(w_{n,1}),$ with simple zeros at the points of sign change of $\mathcal{A}_{n,j}$ on $\Delta_j$ such that
\begin{equation}\label{Amj}
    \mathcal{H}(\C\diff\Delta_{j+1})\ni\frac{\mathcal{A}_{n,j}}{w_{n,j}} = \mathcal{O}\left(\frac{1}{z^{\deg(w_{n,1})+1}}\right).
\end{equation}

For $j=m-1$, we have
\begin{equation}
    \mathcal{H}(\C\diff\Delta_m)\ni\frac{a_{n,m}\widehat{s}_{m,m}-a_{n,m-1}}{w_{n,m-1}}(z) = \mathcal{O}\left(\frac{1}{z^{\deg(w_{n,1})+1}}\right)\label{HP_s_mm},
\end{equation}
and using again Lemma \ref{Remainder_type_i}, we obtain
\begin{equation*}
    \int x^\nu a_{n,m}(x)\frac{\D s_{m,m}(x)}{w_{n,m-1}(x)}=0,\qquad \nu =0,1,\ldots\deg(w_{n,1})-1.
\end{equation*}
Whence, $a_{n,m}$ has at least $\deg(w_{n,1})$ sign changes on $\Delta_m$.  Recall that in general $\deg(w_{n,1}) \geq n -2D$ and its degree is $\geq n-D$ if the zeros of $T$ lie outside $\Delta_1$. This settles the question on the number of sign changes of the forms on the different intervals.

Now let us consider the question of an upper bound on the total number of zeros that $\mathcal{A}_{n,j}, j=0,\ldots,m-1$ may have in $\C \setminus \Delta_{j+1}$. The arguments are pretty much the same. We will play on the fact that $\deg(a_{n,m}) \leq n$ and $a_{n,m} \not\equiv 0$.

Assume that $a_{n,m} \equiv 0$. From \eqref{relint} with $w\equiv 1$ it follows that for each $j=1,\ldots,m-1$
\begin{equation*}
    {\mathcal{A}_{n,j}(z)}  = \int \frac{\mathcal{A}_{n,j+1}(x)}{z-x} {\D \sigma_{j+1}(x)} .
\end{equation*}
Since $\mathcal{A}_{n,m} = (-1)^m a_{n,m}$, this formula with $j= m-1$ readily implies that $a_{n,m-1} \equiv 0$ and $\mathcal{A}_{n,m-1}\equiv 0$ if $a_{n,m} \equiv 0$. Going down on the indices $j$ we conclude that $a_{n,j} \equiv 0$ and $\mathcal{A}_{n,j}\equiv 0$ for all $j=1,\ldots,m$. Formula \eqref{relint} also implies that
\begin{equation*}
    {\mathcal{L}_{n,0}(z)}  = \int \frac{\mathcal{L}_{n,1}(x)}{z-x} {\D \sigma_{1}(x)} .
\end{equation*}
If $\mathcal{A}_{n,1} \equiv 0$ so too ${\mathcal{L}_{n,1}} \equiv 0$; consequently,  ${\mathcal{L}_{n,0}} \equiv 0$ and $a_{n,0} \equiv 0$. In particular, should $a_{n,0}\equiv 0$ then necessarily $a_{n,j}\equiv 0,j=0,\ldots,m$. However, we explicitly excluded the trivial solution in Definition 2. So $a_{n,m} \not \equiv 0$.

Suppose that $\mathcal{A}_{n,0}$ has at least $2D+1$ zeros in $\C \setminus \Delta_1$. Then, there exists a polynomial with real coefficients $w_{n,0}$ of degree $\geq 2D+1$ whose zeros lie in $\C\setminus \Delta_1$ such that
\[
    \frac{\mathcal{L}_{n,0}(z)}{w_{n,0}(z)} = \frac{ T(z)\mathcal{A}_{n,0}(z)}{w_{n,0}(z)} = \mathcal{O}\left(\frac{1}{z^{n+D+2}}\right) \in \mathcal{H}(\C \setminus \Delta_1).
\]
Using  (\ref{Orth_remainder}), we obtain
\begin{equation*}
    \int x^\nu \mathcal{L}_{n,1}(x)\frac{\D \sigma_1(x)}{w_{n,0}(x)} =0,\qquad \nu = 0,1,\ldots,n+D.
\end{equation*}
This means that $\mathcal{L}_{n,1}$ has at least $n+D+1$ sign changes on $\Delta_1$ and $\mathcal{A}_{n,1}$ at least $n+1$ sign changes on $\Delta_1$. Continuing as in the proof of the first part of the lemma we arrive at the conclusion that $a_{n,m}$ has at least $n+1$ sign changes on $\Delta_m$ which is not possible since it is a polynomial of degree $\leq n$ not identically equal to zero. Therefore, $\mathcal{A}_{n,0}$ has at most $2D$ zeros in $\C \setminus \Delta_1$. Notice that when the zeros of $T$ are in the complement of $\Delta_1$ in order to conclude that $\mathcal{A}_{n,1}$ has $n+1$ sign changes on $\Delta_1$ it is sufficient to assume that $\deg(w_{n,0}) \geq D +1$, so in this case one can prove that $\mathcal{A}_{n,0}$ has at most $D$ zeros in $\C \setminus \Delta_1$.

Suppose that $\mathcal{A}_{n,k}$ has at least $n+1$ zeros in $\C \setminus \Delta_{k+1}$ for some specific $k \in \{1,\ldots,m-1\}$ and $n$. Then  there exists a polynomial $w_{n,k}$ with real coefficients of degree $\geq n+1$ such that
\begin{equation*}
     \frac{\mathcal{A}_{n,k}}{w_{n,k}} = \mathcal{O}\left(\frac{1}{z^{n+2}}\right) \in \mathcal{H}(\C\diff\Delta_{k+1}),
\end{equation*}
which, reasoning as above, implies that $\mathcal{A}_{n,k+1}$ has at least $n+1$ sign changes on $\Delta_{k+1}$. Continuing the process one proves that for $j=k+1,\ldots,m,$ the forms $\mathcal{A}_{n,j}$ also have at least $n+1$ sign changes on $\Delta_j$ which contradicts the fact that $a_{n,m}$ cannot have more than $n$ zeros.

Finally, suppose that for some $n$ we know that $a_{n,m}$ has at exactly $n-D$ sign changes on $\Delta_m$ and $\mathcal{A}_{n,k}$ has at least $n-D+1$ zeros in $\C \setminus \Delta_{k+1}$ for some $k \in \{1,\ldots,m-1\}$. Then there exists a polynomial $w_{n,k}$ with real coefficients with zeros in $\C \setminus \Delta_{k+1}$ and degree $\geq n-D+1$ such that
\begin{equation*}
     \frac{\mathcal{A}_{n,k}}{w_{n,k}} = \mathcal{O}\left(\frac{1}{z^{n-D+2}}\right) \in \mathcal{H}(\C\diff\Delta_{k+1}).
\end{equation*}
Repeating the arguments used above it follows that for $j=k+1,\ldots,m,$ the forms $\mathcal{A}_{n,j}$ have at least $n-D+1$ sign changes on $\Delta_j$. In particular, $a_{n,m}$ would have $n-D+1$ sign changes on $\Delta_m$ against our assumption. Thus, $\mathcal{A}_{n,j},j=1,\ldots,m-1$ has at most $n-D$ zeros on $\C \setminus \Delta_{j+1}$. Since it has $n-D$ sign changes on $\Delta_j$ the statement readily follows. That $\mathcal{A}_{n,0}$ has no zeros in $\C \setminus \Delta_1$ is proved analogously.
\end{proof}

\subsection{Convergence in Hausdorff content}

We underline that in the next result no assumption is made on  the rational functions $r_k$ except that they have real coefficients.

\begin{theorem}
\label{Th_conv_Haus}
For each $n \geq 2D$, let $a_{n,0}, a_{n,1},\ldots, a_{n,m}$ be the Hermite-Pad\'e polynomials associated with the Nikishin system $\mathcal{N}(\sigma_1,\ldots,\sigma_m)$ and $(r_1,\ldots,r_m)$ such that (\ref{Problem1}) and (\ref{Problem2}) holds. Suppose that either  $\sigma_m$ satisfies Carleman's conditions or $\Delta_{m-1}$ is a bounded interval. Then,
\begin{equation*}
h-\lim_{n\to \infty} \frac{a_{n,j}}{a_{n,m}} = \widehat{s}_{m,j+1},\qquad h-\lim_{n\to \infty} \frac{a_{n,m}}{a_{n,j}} = \widehat{s}^{-1}_{m,j+1},\qquad j=1,\ldots,m-1,
\end{equation*}
and
\begin{equation}
  h-\lim_{n\to \infty} \frac{a_{n,0}}{a_{n,m}} = f = \widehat{s}_{m,1} - \sum_{k=1}^{m-1} (-1)^{k}\widehat{s}_{m,k+1}r_k - (-1)^m r_m, \label{limit_an0_anm}
\end{equation}
on each compact subset $\mathcal{K}\subset\C\diff\Delta_m$. Moreover, the polynomial $a_{n,j}$, $j=1,\ldots,m-1$, has at least $n-2D-m+j$ sign changes on $\Delta_m$. If the zeros of the polynomial $T$ lie in the complement of $\Delta_1$ then the polynomial $a_{n,j}$, $j=1,\ldots,m-1$, has at least $n-D-m+j$ sign changes in $\Delta_m$.
\end{theorem}
\begin{proof} Let us point out that if $\sigma_m$ satisfies Carleman's condition so do the measures $s_{m,j}$ and $\tau_{m,j}, j=1,\ldots,m$, see \cite[Theorem 1.5]{Lago_Sergio}. We reduce the proof of the limit relations to Lemma \ref{Conv_MP_Incom_Pade}.

Assume that $n \geq 2D$. Notice that  (\ref{HP_s_mm}) tells us that the polynomials $a_{n,m-1}$, $a_{n,m}$ and $w_{n,m-1}$ satisfy the conditions of Definition \ref{Incomplete_Pade}. Therefore, the rational fractions $a_{n,m-1}/a_{n,m}$ form a sequence of incomplete diagonal multi-point Pad\'e approximants of $\widehat{s}_{m,m}$.

Using Lemma \ref{Conv_MP_Incom_Pade} we have  convergence in Hausdorff content on each compact subset of $\C\diff\Delta_m$. That is,
$$ h-\lim_{n}\frac{a_{n,m-1}}{a_{n,m}}(z) = \widehat{s}_{m,m}(z).$$
Dividing $\frac{\mathcal{A}_{n,m-1}}{w_{n,m-1}}$ by $\widehat{s}_{m,m}=\widehat{\sigma}_m$ and using (\ref{CauchyTranInver}), we obtain
$$\frac{(a_{n,m-1}\ell_m-a_{n,m})+a_{n,m-1}\widehat{\tau}_m}{w_{n,m-1}}(z) = \mathcal{O}\left(\frac{1}{z^{n-D}}\right).$$
So, we have again a sequence of incomplete multi-point approximants of $\widehat{\tau}_m$ and,  consequently,
\begin{equation*}
    h-\lim_{n}\left( \ell_m - \frac{a_{n,m}}{a_{n,m-1}}\right)(z) = \widehat{\tau}_{m}(z),
\end{equation*}
which is equivalent to
\begin{equation*}
    h-\lim_{n}\frac{a_{n,m}}{a_{n,m-1}}(z) = \widehat{\sigma}_{m}^{-1}(z)
\end{equation*}
on compact subsets of $\C\diff\Delta_m$.

Now, using (\ref{CauchyTranInver}) and (\ref{QuotCauchyTran}), for $j=1,\ldots,m-2$, we have
\begin{multline*}
    \frac{\mathcal{A}_{n,j}}{\widehat{\sigma}_{j+1}} = \left((-1)^j \ell_{j+1}a_{n,j} + (-1)^{j+1}a_{n,j+1} + \sum_{k=j+2}^{m}(-1)^k \frac{|s_{j+1,k}|}{|\sigma_{j+1}|}a_{n,k}\right)\\ + (-1)^j a_{n,j}\widehat{\tau}_{j+1} - \sum_{k=j+2}^m (-1)^k a_{n,k}\langle \tau_{j+1},\langle s_{j+2,k},\sigma_{j+1}\rangle\widehat{\rangle}.
\end{multline*}
The quotient $\frac{\mathcal{A}_{n,j}}{\widehat{\sigma}_{j+1}}$ has the same structure as $\mathcal{A}$ in Lemma \ref{Remainder_type_i}.
Moreover, using \eqref{Amj}, we obtain
\begin{equation*}
    \frac{\mathcal{A}_{n,j}(z)}{(\widehat{\sigma}_{j+1}w_{n,j})(z)} = \mathcal{O}\left(\frac{1}{z^{n-2D}}\right) \in \mathcal{H}(\C\diff\Delta_{j+1}),
\end{equation*}
and, as consequence of (\ref{Orth_remainder}), for $\nu=0,\ldots,n-2D-2$, it follows that
\begin{equation*}
    0 = \int_{\Delta_{j+1}} x^\nu \left((-1)^j a_{n,j} - \sum_{k=j+2}^m (-1)^k a_{n,k}\langle s_{j+2,k},\sigma_{j+1}\widehat{\rangle}\right)(x) \frac{\D \tau_{j+1}(x)}{w_{n,j}(x)}.
\end{equation*}
The expression in parenthesis under the integral sign has at least $n-2D-1$ sign changes in $\mathring{\Delta}_{j+1}$. Thus, there exists a polynomial $w^*_{n,j}$ of degree $n-2D-1$ whose zeros are simple and lie in $\mathring{\Delta}_{j+1}$ such that
$$\frac{1}{w^*_{n,j}}\left((-1)^j a_{n,j} - \sum_{k=j+2}^m (-1)^k a_{n,k}\langle s_{j+2,k},\sigma_{j+1}\widehat{\rangle}\right)\in\mathcal{H}(\C\diff\Delta_{j+2}).$$

Direct computation or \cite[Lemma 2.1]{Lago_Sergio_Jacek} allows to deduce the formula
\begin{equation}
    \mathcal{A}_{n,j}-\widehat{s}_{j+1,j+1}\mathcal{A}_{n,j+1} = (-1)^j a_{n,j} - \sum_{k=j+2}^m (-1)^k a_{n,k}\langle s_{j+2,k},\sigma_{j+1}\widehat{\rangle}.\nonumber
\end{equation}
From the statement of our problem we know that $\mathcal{A}_{n,j}-\widehat{s}_{j+1,j+1}\mathcal{A}_{n,j+1}$ is $\mathcal{O}\left({1}/{z}\right)$. Hence,
\begin{equation*}
    \frac{1}{w^*_{n,j}(z)}\left((-1)^j a_{n,j} - \sum_{k=j+2}^m (-1)^k a_{n,k}\langle s_{j+2,k},\sigma_{j+1}\widehat{\rangle}\right)(z) = \mathcal{O}\left(\frac{1}{z^{n-2D}}\right) . \qquad z\rightarrow\infty.
\end{equation*}
Notice that if $j=m-2$ we have
\begin{equation*}
    \frac{a_{n,m-2}-a_{n,m}\widehat{s}_{m,m-1}}{w^*_{n,j}}(z) = \mathcal{O}\left(\frac{1}{z^{n-2D}}\right).
\end{equation*}
Thus, $a_{n,m-2}/a_{n,m}$ is an incomplete diagonal multi-point Pad\'e approximant of $\widehat{s}_{m,m-1}$ and we obtain convergence in Hausdorff convergence on compact subsets of $\C\diff\Delta_m$

$$ h-\lim_{n\to \infty}\frac{a_{n,m-2}}{a_{n,m}}(z)= \widehat{s}_{m,m-1}(z).$$
Dividing by $\widehat{s}_{m,m-1}$ and arguing as we did above it also follows that
$$ h-\lim_{n\to \infty}\frac{a_{n,m}}{a_{n,m-2}}(z)= \widehat{s}^{-1}_{m,m-1}(z).$$

Using the identity $\langle s_{j+2,k}, s_{j+1,j+1} \rangle = \langle s_{j+2,j+1},s_{j+3,k}\rangle$ for $k=j+3,\ldots,m$, we deduce
\begin{multline}
    (-1)^j a_{n,j} - \sum_{k=j+2}^m (-1)^k a_{n,k}\langle s_{j+2,k},\sigma_{j+1}\widehat{\rangle}\\ = (-1)^{j}a_{n,j} - (-1)^{j+2}a_{n,j+2}\widehat{s}_{j+2,j+1} - \sum_{k=j+3}^m (-1)^k a_{n,k} \langle s_{j+2,j+1}, s_{j+3,k}\widehat{\rangle}.\label{Prep_Elim}
\end{multline}
As we wish to eliminate $\widehat{s}_{j+2,j+1}$ in the right hand side of (\ref{Prep_Elim}), we divide both sides by it and use again (\ref{CauchyTranInver}) and (\ref{QuotCauchyTran}). Then,
\begin{multline}
    \left( (-1)^j a_{n,j}\ell_{j+2,j+1}-(-1)^{j+2}a_{n,j+2} - \sum_{k=j+3}^m (-1)^k\frac{|\langle s_{j+2,j+1},s_{j+3,k} \rangle|}{|s_{j+2,j+1}|}\right) +\\
    (-1)^j a_{n,j}\widehat{\tau}_{j+2,j+1} + \sum_{k=j+3}^m (-1)^k a_{n,k}\langle \tau_{j+2,j+1}, \langle s_{j+3,k}, s_{j+2,j+1}\rangle\widehat{\rangle}\nonumber
\end{multline}
which is a linear form as those in Lemma \ref{Remainder_type_i}. Thus
$$\mathcal{H}(\C\diff\Delta_{j+2})\ni\frac{1}{(w^*_{n,j}\widehat{s}_{j+2,j+1})}\left((-1)^j a_{n,j} + \sum_{k=j+3}^m (-1)^k a_{n,k}\langle  s_{j+3,k}, s_{j+2,j+1}\widehat{\rangle}\right) = \mathcal{O}\left(\frac{1}{z^{n-2D-1}}\right).$$
Moreover, for $\nu=0,1,\ldots, n-2D-3$,
\begin{equation}
    \int x^\nu \left((-1)^j a_{n,j} + \sum_{k=j+3}^m (-1)^k a_{n,k}\langle  s_{j+3,k}, s_{j+2,j+1}\widehat{\rangle}\right)(x)\frac{\D \tau_{j+2,j+1}(x)}{w^*_{n,j}(x)} =0.\nonumber
\end{equation}
So, the expression in parenthesis has at least $n-2D-2$ sign changes in the interior of $\Delta_{j+2}$, and we can assure the existence of a polynomial $w^*_{n,j+1}$, $\deg w^*_{n,j+1}=n-2D-2$, with simple zeros located at the points of sign change inside $\Delta_{j+2}$ so that
\begin{equation}
    \frac{1}{w^*_{n,j+1}}\left((-1)^j a_{n,j} + \sum_{k=j+3}^m (-1)^k a_{n,k}\langle  s_{j+3,k}, s_{j+2,j+1}\widehat{\rangle}\right)\in\mathcal{H}(\C\diff\Delta_{j+3}).\nonumber
\end{equation}
Using \cite[Lemma 2.1]{Lago_Sergio_Jacek} with $r=j+2$ (or direct calculation), we have
\begin{equation}
 (-1)^j a_{n,j} + \sum_{k=j+3}^m (-1)^k a_{n,k}\langle  s_{j+3,k}, s_{j+2,j+1}\widehat{\rangle}  = \mathcal{A}_{n,j} - \widehat{s}_{j+1,j+1}\mathcal{A}_{n,j+1} + \widehat{s}_{j+1,j+1}\mathcal{A}_{n,j+2},\nonumber
\end{equation}
and from the definition of the forms $\mathcal{A}_{n,j}$ the right hand side is $\mathcal{O}(1/z)$; consequently,
\begin{equation}
    \frac{1}{w^*_{n,j+1}(z)}\left((-1)^j a_{n,j} + \sum_{k=j+3}^m (-1)^k a_{n,k}\langle  s_{j+3,k}, s_{j+2,j+1}\widehat{\rangle}\right)(z)=\mathcal{O}\left(\frac{1}{z^{n-2D-1}}\right).\nonumber
\end{equation}
In particular, if $j=m-3$, it is not difficult to see that the fraction $a_{n,m-3}/a_{n,m}$ is an incomplete diagonal multi-point Pad\'e approximant of $\widehat{s}_{m,m-2}$ from where we can deduce the Hausdorff convergence on compact subsets of $\C\diff\Delta_m$

$$h-\lim_n\frac{a_{n,m-3}}{a_{n,m}} = \widehat{s}_{m,m-2},$$
and similarly
$$h-\lim_n\frac{a_{n,m}}{a_{n,m-3}} = \widehat{s}^{-1}_{m,m-2}.$$

This process can be continued inductively. After $m-j-1$ reductions we obtain the existence of a polynomial $\widetilde{w}_{n,j}$ with degree $\geq n-2D-m+j$ with simple zeros inside $\Delta_{m-1}$ such that
\begin{equation}
\label{sign_changes}
    \frac{a_{n,j}-a_{n,m}\widehat{s}_{m,j+1}}{\widetilde{w}_{n,j}}(z) = \mathcal{O}\left(\frac{1}{z^{n-2D-m+j+2}}\right) \in\mathcal{H}(\C\diff\Delta_m), \qquad z\rightarrow\infty,
\end{equation}
which allows us to deduce that
\begin{equation}
    h-\lim_n \frac{a_{n,j}}{a_{n,m}} = \widehat{s}_{m,j+1},\nonumber
\end{equation}
on compact subsets of $\C\diff\Delta_m$.

As an immediate consequence, we have
\begin{equation}
    \frac{a_{n,j}-a_{n,m}\widehat{s}_{m,j+1}}{\widehat{s}_{m,j+1}\widetilde{w}_{n,j}}(z) = \mathcal{O}\left(\frac{1}{z^{n-2D-m+j+1}}\right) \in\mathcal{H}(\C\diff\Delta_m), \qquad z\rightarrow\infty,\nonumber
\end{equation}
but
\begin{equation}
     \frac{a_{n,j}-a_{n,m}\widehat{s}_{m,j+1}}{\widehat{s}_{m,j+1}} = a_{n,j}\widehat{\tau}_{m,j+1} - (a_{n,m}-\ell_{m,j+1}a_{n,j}).\nonumber
\end{equation}
Hence,
\begin{equation}
    \int x^\nu a_{n,j}(x)\frac{\D\tau_{m,j+1}(x)}{\widetilde{w}_{n,j}(x)} =0, \qquad \nu=0,1,\ldots,n-2D-m+j-1.\nonumber
\end{equation}
Therefore, the polynomial $a_{n,j}$ has at least $n-2D-m+j$ sign changes in $\mathring{\Delta}_m$. Also, we obtain
\begin{equation}
    h-\lim_n \frac{a_{n,m}}{a_{n,j}} = \widehat{s}^{-1}_{m,j+1}\nonumber
\end{equation}
on compact subsets of $\C\diff\Delta_m$.

To find the limit of the sequence $a_{n,0}/a_{n,m}$, $n \geq 0,$ we change a little our previous arguments. It is easy to check that the reasonings above do not change substantially if we consider the linear forms $\mathcal{L}_{n,j} := T(z)\mathcal{A}_{n,j}(z)$ instead of $\mathcal{A}_{n,j}$. The main differences are in the asymptotic orders and in the bounds for the number of sign changes in $\Delta_m$, but not in the conclusions.

In consequence, the following holds (see \eqref{ps})
\begin{equation}
    \frac{p_{n,0}-p_{n,m}\widehat{s}_{m,1}}{\widetilde{w}_{n,0}}(z) = \mathcal{O}\left(\frac{1}{z^{n-2D-m+j-1}}\right)\in \mathcal{H}(\C \setminus \Delta_m), \nonumber
\end{equation}
and we conclude that
\begin{equation}
   h-\lim_n \frac{p_{n,0}}{p_{n,m}}= \widehat{s}_{m,1}.\nonumber
\end{equation}
However,
\begin{equation}
    \frac{p_{n,0}}{p_{n,m}} = \frac{a_{n,0}T + \sum_{k=1}^m (-1)^k a_{n,k}Tr_k}{  a_{n,m}T} =  \frac{a_{n,0}}{a_{n,m}} + \sum_{k=1}^m (-1)^{k}\frac{a_{n,k}}{a_{n,m}}r_k.\nonumber
\end{equation}
Therefore, (\ref{limit_an0_anm}) readily follows.

Throughout the proof, if the zeros of   $T$ lie outside $\Delta_1$ then in the right hand side of \eqref{Amj} we can write $\mathcal{O}\left( {1}/{z^{n-D+1}}\right)$  and we can replace $2D$ with $D$ obtaining $n-D-m+j$ sign changes on $\Delta_m$ for $a_{n,j}$ as indicated in the final statement.
\end{proof}

\subsection{Proof of the main result}

\noindent
{\bf Proof of Theorem \ref{Th_conv_unif}}. In the hypothesis of this theorem the zeros of the polynomial $T$ lie outside $\Delta_1$; consequently, according to the last statement of Lemma \ref{Zeros_a_nm} the rational functions  $\frac{a_{n,0}}{a_{n,m}}$ have at most $D$ poles in $\C\diff\Delta_m$. On the other hand, we are assuming that $f$ has exactly $D$ poles in $\C \setminus \Delta_m$. From (\ref{limit_an0_anm}) and Gonchar's lemma \cite[Lemma 1]{gonchar}, we obtain that for all sufficiently large $n\in \N$ the fractions $\frac{a_{n,0}}{a_{n,m}}$ have exactly $D$ poles outside $\Delta_m$. Moreover, Gonchar's lemma asserts that each pole of $f$ in $\C \setminus \Delta_m$ attracts as many zeros of $a_{n,m}$ as its order; that is, if $\zeta \in \C \setminus \Delta_m$ is a pole of $f$ of order $\tau$ then for each $\varepsilon > 0$, there exists $n_0(\tau)\in \N$ such that for all $n\geq n_0(\tau)$ the polynomial $a_{n,m}$ has exactly $\tau$ zeros in the disk $\{z: |z-\zeta| < \varepsilon\}$. Thus the statements about the zeros of $a_{n,m}$ take place.

Fix $\varepsilon > 0$ and let $D_{\varepsilon}$ be $\C \setminus \Delta_m$ minus an $\varepsilon$ neighborhood of each pole of $f$ in this region. Then, there exists $n_0$ such that for all $n \geq n_0$ and $j=0,\ldots,m-1$ the rational functions $a_{n,j}/a_{n,m}$ are analytic in $D_{\varepsilon}$. From \cite[Lemma 1]{gonchar} it follows that the limits in Lemma \ref{Th_conv_Haus} take place uniformly on each compact subset of $D_{\varepsilon}$. Since $\varepsilon > 0$ is arbitrary, the limits in Theorem \ref{Th_conv_unif} hold.

Fix $j=1,\ldots,m-1$. Let $\zeta$ be a zero of $T$ of multiplicity $\tau$. Choose $\varepsilon>0$ small enough and $N$ sufficiently large such that $a_{n,m}$ has no zero on $\{|z-\zeta|=\varepsilon\}$ and exactly $\tau$ zeros inside the circle $\{|z-\zeta|=\varepsilon\}$ for $n\geq N$. As the function $\widehat{s}_{m,j+1}$ is holomorphic  and has no zeros in $\C\diff\Delta_m$, by the uniform convergence we get
\begin{equation}
    \lim_{n\to \infty} \int_{|z-\zeta|=\varepsilon}\frac{(a_{n,j}/a_{n,m})'}{a_{n,j}/a_{n,m}}(z)\D z = \int_{|z-\zeta|=\varepsilon}\frac{(\widehat{s}_{m,j+1})'}{\widehat{s}_{m,j+1}}(z)\D z =0.\nonumber
\end{equation}
Since $a_{n,m}$ has exactly $\tau$ zeros inside $\{|z-\zeta|=\varepsilon\}$ for all sufficiently large $n$, from the argument principle we obtain that $a_{n,j}, j=,\ldots,m-1$ also has exactly $\tau$ zeros inside that disk for all sufficiently large $n$.

Thus, in the circle $\{|z-\zeta|<\varepsilon\}$ the number of zeros of $a_{n,j}$ and $a_{n,m}$ coincide, i.e. $\zeta$ attracts as many zeros of $a_{n,j}$ as its order. We can extend this idea to a smooth Jordan curve $\Gamma$ that surrounds all zeros of $T$ and lies in $\C\diff\Delta_m$. Then $D$ zeros of $a_{n,j}$ accumulate at the zeros of $T$ counting multiplicities and the remaining ones accumulate on $\Delta_m\cup\{\infty\}$. \hfill $\Box$

Theorem \ref{Th_conv_unif} has some consequences on the convergence of the forms $\mathcal{A}_{n,j}$.

\begin{corollary} \label{conforms} Under the assumptions of Theorem \ref{Th_conv_unif}, we have
\[ \lim_{n\to \infty} \frac{\mathcal{A}_{n,j}}{a_{n,m}} = 0, \qquad j=0,\ldots,m-1,
\]
uniformly on each compact subset of $\C \setminus (\Delta_{j+1}\cup \Delta_m \cup \{z:T(z) =0\})$.

\end{corollary}
\begin{proof} From Theorem \ref{Th_conv_unif}  and the expression of the forms $\mathcal{A}_{n,j}$ it follows that for $j=1,\ldots,m-1,$
\[ \lim_{n\to \infty} \frac{\mathcal{A}_{n,j}}{a_{n,m}} =  (-1)^j \widehat{s}_{m,j+1} +  \sum_{k=j+1}^{m-1} (-1)^k \widehat{s}_{m,k+1}\widehat{s}_{j+1,k} + (-1)^m \widehat{s}_{j+1,m} \equiv 0,
\]
uniformly on compact subsets of $\C \setminus (\Delta_{j+1}\cup \Delta_m \cup \{z:T(z) =0\})$. The equivalence to zero of the last expression is a consequence of a well known formula appearing in \cite[Lemma 2.9]{ulises_lago_2}. Similarly,
\[ \lim_{n\to \infty} \frac{\mathcal{A}_{n,0}}{a_{n,m}} =  f +  \sum_{k=1}^{m-1} (-1)^k \widehat{s}_{m.k+1}(\widehat{s}_{1,k} + r_k)+ (-1)^m (\widehat{s}_{1,m} + r_m) \equiv 0,
\]
uniformly on compact subsets of $\C \setminus (\Delta_{1}\cup \Delta_m \cup \{z:T(z) =0\})$. In proving the equality to zero aside from the identity in
\cite[Lemma 2.9]{ulises_lago_2} one uses the expression of $f$.
\end{proof}

\subsection{Rate of convergence}

Throughout this subsection we assume that the conditions of Theorem \ref{Th_conv_unif} are in place. We will begin showing that when $\Delta_m$ is a finite interval then convergence takes place with geometric rate. We will derive this result using Theorem \ref{Th_conv_unif} and the maximum principle. Initially we need to introduce some concepts that will be needed.

Let $\varphi_t, t \in \overline{\C} \setminus \Delta_m, $ be the conformal representation of $\overline{\C} \setminus \Delta_m$ onto $\{w: |w| < 1\}$ such that $\varphi_t(t)= 0, \varphi_t^\prime (t) > 0$. It is easy to verify that $|\varphi_t(z)|$ can be extended continuously to $\overline{\C}^2$ in the two variables $z,t$ and equals zero only when $z=t$. In fact
\[ |\varphi_t(z)| = \left|\frac{\varphi_\infty(z)- \varphi_\infty(t)}{1 - \overline{\varphi_\infty(t)}\varphi_\infty(z)}\right|.
\]
Let $0 < \rho < 1$ and
\[ \gamma_\rho := \{z: |\varphi_\infty(z)|= \rho\}.
\]
Fix a compact set $\mathcal{K} \subset \overline{\C}  \setminus (\Delta_m \cup \{z: T(z) =0\})$.
Take $\rho$ sufficiently close to $1$ so that $\mathcal{K}$ lies in the unbounded connected component of the complement of $\gamma_\rho$.
Set
\begin{equation} \label{delta} \kappa_\rho := \inf\{|\varphi_t(z)|: t \in \Delta_{m-1}, z \in \gamma_\rho\}, \qquad \delta(\mathcal{K}) = \max\{|\varphi_t(z)|: t \in \Delta_{m-1}, z \in \mathcal{K}\}.
\end{equation}
From the continuity of $|\varphi_t(z)|$ in the two variables it readily follows that
\[ \lim_{\rho \to 1} \kappa_\rho = 1, \qquad \delta(\mathcal{K}) < 1.
\]
As usual, $\|\cdot\|_{\mathcal{K}}$ denotes the uniform norm on $\mathcal{K}$.

\begin{corollary} \label{cor:4}
Under the hypothesis of Theorem \ref{Th_conv_unif} if we assume additionally that $\Delta_m$ is bounded then
\begin{equation} \label{rate3}
\limsup_n \left\|\frac{a_{n,j}}{a_{n,m}} - \widehat{s}_{m,j+1}\right\|_{\mathcal{K}}^{1/n} \leq \delta(\mathcal{K})\|\varphi_\infty\|_{\mathcal{K}}< 1, \qquad j=1,\ldots,m-1,
\end{equation}
and
\begin{equation} \label{polos*}
  \limsup_n \left\|\frac{a_{n,0}}{a_{n,m}} - f\right\|_{\mathcal{K}}^{1/n} \leq \delta(\mathcal{K})\|\varphi_\infty\|_{\mathcal{K}}< 1
\end{equation}
for every compact set $\mathcal{K} \subset \C \setminus (\Delta_m \cup \{z: T(z) =0\})$ and $\delta(\mathcal{K})$ is the quantity defined in \eqref{delta}.
\end{corollary}

\begin{proof}
Fix $\mathcal{K} \subset \overline{\C} \setminus (\Delta_m \cup \{z: T(z)=0\})$. According to Theorem \ref{Th_conv_unif}, for all sufficiently large $n > N$ the polynomials $a_{n,m}$ have exactly $D$ zeros in $\C \setminus \Delta_m$ and they lie at a positive distance from $\mathcal{K}$ (independent of $n > D$. In the sequel we only consider such $n$'s.

Let $q_{n,m} = \prod_1^D (z - x_{n,k})$ be the monic polynomial of degree $D$ whose zeros are the roots of $a_{n,m}$ outside $\Delta_m$. From Theorem \ref{Th_conv_unif} we know that $\lim_{n\to \infty}q_{n,m} = T$.
Fix $j=1,\ldots,m$. Assume that $\widetilde{w}_{n,j}(z) = \prod_{k=1}^{\deg(\widetilde{w}_{n,j})}(z - \zeta_{n,j,k})$, where $\widetilde{w}_{n,j}$ is the polynomial introduced in the proof of Theorem \ref{Th_conv_Haus} (see \eqref{sign_changes}). Set
\[ \varphi_{n,j}(z) := \prod_{k=1}^{\deg(\widetilde{w}_{n,j})} \varphi_{\zeta_{n,j,k}}(z), \qquad \psi_n(z) := \prod_{k=1}^D \varphi_{x_{n,k}}(z).
\]

From \eqref{sign_changes} it follows that
\[ \psi_n \frac{(a_{n,j}/a_{n,m}) - \widehat{s}_{m,j+1}}{\varphi_\infty^n\varphi_{n,j}} \in\mathcal{H}(\overline{\C} \diff\Delta_m).
\]
Take $\rho$ sufficiently close to $1$ so that $\mathcal{K}$ lies in the unbounded connected component of the complement of $\gamma_\rho$. On $\gamma_\rho$, for all sufficiently large $n > N_1 \geq N$, we have
\begin{equation} \label{rate} \left\|\psi_n \frac{(a_{n,j}/a_{n,m}) - \widehat{s}_{m,j+1}}{\varphi_\infty^n\varphi_{n,j}}\right\|_{\gamma_\rho} \leq  \rho^{-n}\kappa_\rho^{-\deg{\widetilde{w}_{n,j}}},
\end{equation}
Indeed, $|\psi_{n}(z)| \leq 1$ for all $z \in \overline{\C} \setminus \Delta_m$, $\varphi_{\zeta_{n,j,k}}(z) \geq \kappa_\rho$ for all $\zeta_{n,j,k} \in \Delta_{m-1}$, and for all sufficiently large $n \geq N_2 \geq N_1$, $\|(a_{n,j}/a_{n,m}) - \widehat{s}_{m,j+1}\|_{\gamma_\rho} \leq 1$ since by Theorem \ref{Th_conv_unif} the function under the norm sign converges to zero on $\gamma_\rho$.

Using the maximum principle, from \eqref{rate} it follows that for all $z \in \mathcal{K}$
\begin{equation} \label{rate2}
\left|\frac{a_{n,j}(z)}{a_{n,m}(z)} - \widehat{s}_{m,j+1}(z)\right| \leq \frac{|\varphi_{n,j}(z)}{|\psi_n(z)|}\frac{\varphi_\infty^n(z)}{\rho^n\kappa^{\deg(\widetilde{w}_{n,j})}_\rho} \leq \frac{\|\varphi_\infty\|^n_{\mathcal{K}}}{|\psi_n(z)|\rho^n}\left(\frac{\delta(\mathcal{K})}{\kappa_\rho}\right)^{\deg(\widetilde{w}_{n,j})}
\end{equation}
Since the points $x_{n,1},\ldots,x_{n,D}$ remain bounded away from $\mathcal{K}$ independent of $n$, we obtain that
\[ \inf_{n > N_3} \{|\psi_n(z)|: z \in \mathcal{K}\} \geq C >0,
\]
where $N_3 \geq N_2$ is sufficiently large. On the other hand, recall tnat $n-2D-m+j \leq \deg(\widetilde{w}_{n,j}) \leq n$; consequently, using
\eqref{rate2}, we obtain
\[\limsup_n \left\|\frac{a_{n,j}}{a_{n,m}} - \widehat{s}_{m,j+1}\right\|_{\mathcal{K}}^{1/n} \leq  \frac{\delta(\mathcal{K})\|\varphi_\infty\|_{\mathcal{K}}}{\rho^n\kappa_\rho}.
\]
From here we get \eqref{rate3} since $\lim_{\rho \to 1} \kappa_\rho = 1$.

The proof of \eqref{polos*} is basically the same and is left to the reader.
\end{proof}

We wish to point out that if $\Delta_{m}$ is unbounded but $\Delta_{m-1}$ is bounded then it is also possible to prove convergence with geometric rate modifying slightly the arguments. Of course, the estimate of the rate of convergence will differ from the one above.

\begin{corollary} \label{cor5}
Under the hypothesis of Theorem \ref{Th_conv_unif} if we assume additionally that $\Delta_m$ is bounded then
\[ \limsup_{n\to \infty}\left\|\frac{\mathcal{A}_{n,j}}{a_{n,m}}\right\|_\mathcal{K}^{1/n} \leq \delta(\mathcal{K})\|\varphi_\infty\|_{\mathcal{K}}, \qquad j=0,\ldots,m-1,
\]
for every compact $\mathcal{K} \subset \overline{\C} \setminus (\Delta_{j+1} \cup \Delta_m \cup \{z: T(z) =0\})$.
\end{corollary}

\begin{proof} Indeed, for $j=1,\ldots,m-1$,
\[\frac{\mathcal{A}_{n,j}}{a_{n,m}} =    (-1)^j \frac{a_{n,j}}{a_{n,m}} + \sum_{k=j+1}^m (-1)^k \frac{a_{n,k}}{a_{n,m}}\widehat{s}_{j+1,k}   =
\]
\[
(-1)^j \left(\frac{a_{n,j}}{a_{n,m}} - \widehat{s}_{m,j+1}\right)+ \sum_{k=j+1}^{m-1} (-1)^k \left(\frac{a_{n,k}}{a_{n,m}} - \widehat{s}_{m,k+1}\right)\widehat{s}_{j+1,k}
 \]
 because, according to \cite[Lemma 2.9]{ulises_lago_2}
 \[ (-1)^j \widehat{s}_{m,j+1} +  \sum_{k=j+1}^{m-1} (-1)^k \widehat{s}_{m.k+1}\widehat{s}_{j+1,k} + (-1)^m \widehat{s}_{j+1,m} \equiv 0
 \]
 for all $z \in \C \setminus (\Delta_{j+1} \cup \Delta_m)$. Now it remains to use \eqref{rate3} and trivial estimates. The proof for $j=0$ is similar and it is left to the reader.
\end{proof}

When the measures generating the Nikishin system are regular in the sense of Stahl and Totik, see \cite{Stahl_Totik}, then more precise estimates of the rate of convergence may be given. This will be the subject of a forthcoming paper.

\end{document}